\documentclass[a4paper]{article}
\pdfoutput=1
\usepackage{amsmath}
\usepackage{amsthm}
\usepackage{amssymb}
\usepackage{mathtools}
\usepackage{hyperref}
\usepackage{tikz}

\newcommand{\ootpi}{\frac{1}{2\pi i}}
\newcommand{\V}{\mathcal{V}}
\newcommand{\Vbar}{\overline{\V}}
\newcommand{\Vbarbb}{\overline{\V}_{bb}}
\newcommand{\C}{\mathbf{C}}
\newcommand{\N}{\mathbf{N}}
\newcommand{\Z}{\mathbf{Z}}
\newcommand{\End}{\mathrm{End}}
\newcommand{\after}{\circ}
\newcommand{\blank}{\text{\textendash}}
\newcommand{\labelledMapMapsto}[5]{%
\begin{align*}%
  #1 : #2 &\longrightarrow #3\\%
       #4 &\longmapsto #5%
\end{align*}%
}
\newcommand{\displayperiod}{\,\text{.}}
\newcommand{\displaycomma}{\,\text{,}}
\DeclareMathOperator{\id}{id} 

\newcommand{\vac}{|0\rangle}

\theoremstyle{definition}
\newtheorem{definition}{Definition}[section]
\newtheorem{remark}[definition]{Remark}

\theoremstyle{plain}
\newtheorem{corollary}[definition]{Corollary}
\newtheorem{proposition}[definition]{Proposition}

\newtheorem{theorem}[definition]{Theorem}
\newtheorem*{theorem*}{Theorem}

\title{Geometric Vertex Algebras}
\author{Daniel Bruegmann}
\date{}
\begin{document}
\maketitle
\begin{abstract}
  Geometric vertex algebras are a simplified version of Huang's geometric vertex operator algebras.
  We give a self-contained account of the equivalence of geometric vertex algebras with~$\Z$-graded vertex algebras.
\end{abstract}
A theorem of Huang~\cite{HuangTwoDimConfGeomAndVOAs} establishes the equivalence of vertex operator algebras and geometric vertex operator algebras.
We give a self-contained account of this theorem in the simplified setting of~$\Z$-graded vertex algebras and geometric vertex algebras, that is, without the infinitesimal conformal symmetries given by the Virasoro algebra.

Both geometric vertex algebras and~$\Z$-graded vertex algebras have an underlying~$\Z$-graded vector space~$\V = \bigoplus_{k\in \Z} \V_k$ over~$\C$. 
In a geometric vertex algebra, the multiplication maps~$\mu$ take elements~$a_1,\ldots,a_n \in \V$ placed at pairwise distinct points~$z_1,\ldots,z_n \in \C$, and the product~$\mu(a_1,\ldots, a_n)(z_1,\ldots, z_n)$ is an element of~$\Vbar:= \prod_{k \in \Z} \V_k$ which contains~$\V$ as a subspace.
Since the result of the multiplication is an infinite sequence of elements of~$\V$ instead of a single element of~$\V$, associativity is not formulated in the usual way but instead involves an infinite sum. 
Furthermore,~$\mu(a)(z)$ holomorphically depends on~$z=(z_1,\ldots,z_n)$, and behaves meromorphically as the~$z_i$ come closer to each other.
By contrast, the corresponding vertex algebra contains the Laurent expansions of the functions~$\mu(a,z,b,0) \in \Vbar$ of~$z\in \C\setminus \{0\}$ for all~$a,b\in \V$, whose coefficients turn out be elements of~$\V$.
\begin{theorem*}
The category of~$\Z$-graded vertex algebras is equivalent to the category of geometric vertex algebras.
\end{theorem*}
\paragraph{Related Work.}
The analogous theorem of Huang is Theorem~5.4.5 in~\cite{HuangTwoDimConfGeomAndVOAs}.
The meaning of the above theorem is very close to that of Theorem~2.12 in Runkel's lecture notes~\cite{RunkelCFTandVOAs}.
The holomorphic integral scale covariant field theories of~\cite{RunkelCFTandVOAs} are very similar to geometric vertex algebras.
The main difference is the kind of convergence in the infinite sum in the associativity axiom.
The definition of geometric vertex algebras in this article uses normally convergent sums in its associativity axiom, as opposed to pointwise convergence as in~\cite{RunkelCFTandVOAs}, seen to imply pointwise absolute convergence there. 
We show that the geometric vertex algebra constructed from a~$\Z$-graded vertex algebra always has normally convergent sums in its associativity property.

Costello and Gwilliam construct~$\Z$-graded vertex algebras from certain holomorphic factorization algebras on~$\C$ in~\cite{CostelloGwilliam}. 
Our exposition of obtaining a vertex algebra from a geometric vertex algebra is modeled on their work.
\paragraph{Acknowledgments.}
The author would like to thank his advisors Peter Teichner and Andr\'e Henriques for their support and suggestions. 
Thanks to Bertram Arnold, Achim Krause and Katrin Wendland for discussions.
Thanks to Andr\'e Henriques and Eugene Rabinovich for their feedback on a draft.
This work was carried out at the Max Planck Institute for Mathematics in Bonn and financially supported by its IMPRS on Moduli Spaces.
Thanks to an anonymous referee who made a large number of helpful suggestions.
\tableofcontents
\section{Holomorphic Functions and Normal Convergence}
Unless stated otherwise, all vector spaces are over the field~$\C$ of complex numbers.
Geometric vertex algebras have an~$n$-ary operation parametrized by 
~$z\in \C^n \setminus \Delta$ where
\[
 \Delta = \Delta_n = \{ (z_1,\ldots, z_n) \in \C^n \mid z_i = z_j \text{ for some } i,j \in \{1,\ldots, j\} \text{ with } i\neq j\}\displayperiod
\]
for every natural number~$n \geq 0$.
We write~$\Delta$ instead of~$\Delta_n$ when~$n$ is apparent from the context. 
Before giving the definition of a geometric vertex algebra, we explain the relevant notions of holomorphicity and convergence of infinite sums of holomorphic functions.
Both are based on considering finite-dimensional subspaces of the~$\V_k$ for all~$k \in \Z$ of a~$\Z$-graded vector space~$\V$, a reflection of the rather algebraic nature of vertex algebras.
\begin{definition}
  Let~$U \subseteq \C^n$ be open and~$X$ be a vector space.
  A map~$f:U \rightarrow X$ is \emph{holomorphic} if~$f$ is locally a holomorphic function with values in a finite-dimensional subspace of~$X$. This means that every point~$p \in U$ has an open neighborhood~$V$ together with a finite-dimensional subspace~$Y \subseteq X$ with~$f(V) \subseteq Y$ and~$f|_V : V \rightarrow Y$ holomorphic.
  If~$\V$ is a~$\Z$-graded vector space, then~$\Vbar := \prod_{k \in \Z} \V_k$ and~$p_k:\Vbar \rightarrow \V_k$ denotes the projection for~$k \in \Z$, and~$\mathcal{O}(U;\Vbar)$ denotes the vector space of~$\Vbar$-valued functions~$f$ on~$U$ each of whose components~$p_k \after f$ is holomorphic.
\end{definition}
Note that if~$Y'$ is another subspace of~$X$ such that~$f(V) \subseteq Y'$, then~$f|_V$ is holomorphic as a map to~$Y'$ if and only if~$f|_V$ is holomorphic as a map to~$Y$.
The following proposition is a version of the identity theorem for holomorphic functions in our sense. Its proof uses the Taylor series of holomorphic functions, which can be defined locally and whose coefficients are again holomorphic functions~$U \rightarrow X$.
See \cite[Chapter I, A, Theorem 6]{GunningRossi} for the identity theorem in the case of~$X=\C$, and a more detailed version of the proof given below.
\begin{proposition}[Identity Theorem]
  Let~$U \subseteq \C^n$ be connected and~$f:U \rightarrow X$ be holomorphic.
  If~$f^{-1}(0)$ has non-empty interior, then~$f=0$.
\end{proposition}
\begin{proof}
Let~$I$ be the interior of~$f^{-1}(0)$.
The set~$I$ is equal to the set of points at which the Taylor series of~$f$ vanishes because~$f$ is holomorphic.
The coefficients of the Taylor series of~$f$ are holomorphic functions on~$U$, so their zero sets are closed subsets of~$U$.
Hence~$I$ is closed as the intersection of closed sets.
Since~$I\subseteq U$ is closed, open and non-empty, and~$U$ is connected, it follows that~$U=I \subseteq f^{-1}(0)$.
\end{proof}
If~$f: U \rightarrow X$ is holomorphic, then~$f(K)$ is contained in a finite-dimensional subspace of~$X$ for all compact~$K \subseteq U$. 
As a corollary to the previous proposition, a holomorphic function on a connected set like~$\C^n \setminus \Delta_n$ globally takes values in a finite-dimensional subspace of~$X$.
\begin{corollary}
Let~$U\subseteq \C^n$ be open,~$X$ be a vector space,~$f: U \rightarrow X$ be holomorphic.
If~$U$ is connected, then~$f(U)$ is contained in a finite-dimensional subspace of~$X$.
\end{corollary}
\begin{proof}
Let~$p \in U$ and let~$V$ be an open neighborhood of~$p$ in~$U$ such that~$f(V)$ is contained in a finite-dimensional subspace~$Y$.
Let~$\overline{f}$ be the composite of~$f$ and the quotient map~$X \rightarrow X/Y$.
It follows that~$\overline{f}$ is holomorphic because locally~$\overline{f}$ is the composite of~$f$ and a quotient map between finite-dimensional vector spaces.
The holomorphic function~$\overline{f}$ is zero on~$V$ by construction.
The identity theorem for holomorphic functions implies that~$\overline{f} = 0$, so that~$f(U) \subseteq Y$.
\end{proof}
Recall from~\cite[Part A, Chapter 3]{Remmert} that a series~$\sum_{i\in I} f_i$ of holomorphic functions on an open~$U \subseteq \C^n$ is called \emph{normally convergent} if every~$p\in U$ has a neighborhood~$V$ such that~$\sum_{i\in I} ||f||_V < \infty$, where~$||f||_V = \sup_{q\in V} |f(q)|$ is the supremum semi-norm of~$f$ on~$V$.
Normal convergence is equivalent to~$\sum_{i \in I} ||f_i||_K <\infty$ for all compact~$K\subseteq U$, and this is the condition we generalize to our setting.
\begin{definition}
Let~$U \subseteq \C^n$ and~$X$ be a vector space.
A sum~$\sum_{i\in I} f_i$ of holomorphic~$X$-valued functions on an open~$U \subseteq \C^n$ is \emph{normally convergent} if, for all compact~$K\subseteq U$, there exists a finite-dimensional subspace~$Y$ such that~$f_i(K) \subseteq Y$ for all~$i\in I$ and~$\sum_{i \in I} ||f_i||_K < \infty$, where the supremum semi-norm~$||f_i||_K = \sup_{x\in K} ||f_i(x)||_Y$ refers to one of the equivalent norms~$||\blank||_Y$ on~$Y$.
\end{definition}
\section{Geometric Vertex Algebras}
\begin{definition}
  A \emph{geometric vertex algebra} consists of:
\begin{itemize}
  \item A $\Z$-graded vector space~$\V=\bigoplus_{k\in \Z} \V_k$ over~$\C$. 
  \item Linear maps~$\mu: \V^{\otimes n} \rightarrow \mathcal{O}(\C^n \setminus \Delta; \Vbar)$ for~$n\geq 0$ where~$\Vbar = \prod_{k \in \Z} \V_k$.  
For all~$n$, we write
\[
  \mu(a)(z) = \mu(a,z) = \mu(a_1,z_1,\ldots,a_n,z_n)
\] for the value of the function corresponding to~$a =a_1 \otimes \ldots \otimes a_n \in \V^{\otimes n}$ at~$z=(z_1,\ldots,z_n)\in \C^n \setminus \Delta$.
\end{itemize}
  The axioms of a geometric vertex algebra are:
  \begin{itemize} 
    \item (permutation invariance) 
\[ \mu(a^\sigma,z^\sigma) = \mu(a,z)\] for~$a \in \V^n$,~$z\in \C^n \setminus \Delta$ and every permutation~$\sigma \in \Sigma_n$. 
    \item (insertion at zero) \[\mu(a,0) = a\] for all~$a\in \V$, where~$a$ is viewed as element of~$\Vbar$ via the embedding 
    \[\V = \bigoplus_n \V_n \hookrightarrow \prod_n \V_n = \Vbar\displayperiod \]
    \item (associativity)
For~$a_1, \ldots, a_m \in \V,$ $b_1,\ldots, b_n \in \V$,~$z \in \C^{m+1} \setminus\Delta$, $w \in \C^n \setminus\Delta$ with~$\max_i |w_i| < \min_{1\leq j \leq m} |z_j -z_{m+1}|$, and~$l\in\Z$,
\begin{align}
&\sum_{k \in \Z} p_l\mu(a_1,z_1,\ldots, a_m,z_m, p_k\mu(b_1, w_1, \ldots, b_n, w_n),z_{m+1}) \label{display:AssociativityLHS}\\
 &= p_l\mu(a_1,z_1,\ldots, a_m,z_m, b_1, w_1+z_{m+1}, \ldots, b_n, w_n+ z_{m+1}) \label{display:AssociativityRHS}
\end{align}
where the sum of~$\V_l$-valued functions of~$z$ and~$w$ is normally convergent.
Here~$m\geq 0$ and~$n\geq 0$. 
  \item ($\C^\times$-equivariance) For all~$z\in \C^\times$,~$a_1, \ldots a_n\in \V$ and~$w \in \C^n \setminus\Delta$
    \[
       z.\mu(a_1,w_1, \ldots, a_n, w_n) = \mu(z.a_1, zw_1, \ldots, z.a_n, zw_n)
    \]
    where~$z\in \C^\times$ acts on~$\Vbar$ and~$\V$ via multiplication by~$z^l$ on~$\V_l$ for all~$l \in \Z$.  

  \item (meromorphicity) For all~$a,b\in\V$, there exists an integer~$N$ such that the function~$z^N\mu(a,z,b,0)$ of~$z\in \C \setminus \{ 0\}$ extends holomorphically to a function of~$z\in \C$.
\end{itemize}
\end{definition}
In the associativity axiom of a geometric vertex algebra, the condition that \[\max_i |w_i| < \min_{1\leq j \leq m} |z_j -z_{m+1}| \] is equivalent to saying that all~$w_j +z_{m+1}$,~$j=1,\ldots,n$, are contained in the largest open ball around~$z_{m+1}$ not containing any of the~$z_i$ for~$i=1,\ldots,m$.
\begin{figure}
  \begin{center}
  \begin{tikzpicture}
    \fill (3,3) circle (2pt) node[below] {$z_5$};
    \fill (3.8,3.9) circle (2pt) node[right] {$z_1$};
    \fill (0,3.5) circle (2pt) node[below] {$z_2$};
    \fill (-1,1.5) circle (2pt) node[above] {$z_3$};
    \fill (2,1) circle (2pt) node[left] {$z_4$};
    \draw[dashed] (3,3) circle (1);
    \fill (3.6,2.7) circle (2pt) node[below right] {$w_3 + z_5$};
    \fill (3,3.7) circle (2pt) node[above left] {$w_1 + z_5$};
    \fill (2,3) circle (2pt) node[below left] {$w_2 + z_5$};
  \end{tikzpicture}
  \end{center}
  \caption{In the associativity axiom of a geometric vertex algebra, the~$w_i+z_{m+1}$ have to be closer to $z_{m+1}$ than any of the other~$z_j$ for $1\leq j \leq m$.}
  \label{figure:PrerequisiteInAssociativityAxiom}
\end{figure}
See Figure~\ref{figure:PrerequisiteInAssociativityAxiom}.
Leaving out the projection~$p_l$ in~(\ref{display:AssociativityLHS}) and~(\ref{display:AssociativityRHS}) and taking normal convergence of sums in~$\Vbar$ to mean normal convergence in each~$\V_l, l\in\Z$, associativity says that
\begin{align}
&\sum_{k \in \Z} \mu(a_1,z_1,\ldots, a_m,z_m, p_k\mu(b_1, w_1, \ldots, b_n, w_n),z_{m+1}) \label{display:AssociativityLHSinVbar}\\
 &= \mu(a_1,z_1,\ldots, a_m,z_m, b_1, w_1+z_{m+1}, \ldots, b_n, w_n+ z_{m+1}). \label{display:AssociativityRHSinVbar}
\end{align}
We use the sum in~(\ref{display:AssociativityLHSinVbar}) as the definition of repeated multiplication in a geometric vertex algebra when the associativity axiom implies that this sum converges.
\begin{definition}\label{definition:RepeatedMu}
Let~$(\V,\mu)$ be a geometric vertex algebra.
For~$a_1, \ldots, a_m \in \V,$ $b_1,\ldots, b_n \in \V$,~$z \in \C^{m+1} \setminus\Delta$, $w \in \C^n \setminus\Delta$ with~$\max_i |w_i| < \min_{1\leq j \leq m} |z_j -z_{m+1}|$, let
\begin{align*}
&\mu(a_1,z_1,\ldots, a_m,z_m, \mu(b_1, w_1, \ldots, b_n, w_n),z_{m+1})\\
& := \sum_{k \in \Z} \mu(a_1,z_1,\ldots, a_m,z_m, p_k\mu(b_1, w_1, \ldots, b_n, w_n),z_{m+1}) \in \Vbar,
\end{align*}
so that associativity can be written as
\begin{align*}
&\mu(a_1,z_1,\ldots, a_m,z_m, \mu(b_1, w_1, \ldots, b_n, w_n),z_{m+1})\\
& = \mu(a_1,z_1,\ldots, a_m,z_m, b_1, w_1+z_{m+1}, \ldots, b_n, w_n+ z_{m+1}).
\end{align*}
\end{definition}

By permutation invariance the multiplication~$\mu$ only depends on the set of pairs
\[
\{(a_1,z_1),\ldots,(a_m,z_m) \}
\]
for~$a_1,\ldots, a_m \in \V$ and~$z \in \C^m \setminus\Delta$ and not on the ordering of the pairs.
\begin{definition}
The image~$\mu(\emptyset)$ of~$1\in \C$ under the multiplication map
\[\C = \V^{\otimes 0} \rightarrow \mathcal{O}(\mathrm{pt};\Vbar) \cong \Vbar\]
for~$n=0$ is called the \emph{vacuum vector}~$\vac$ of~$\V$, or unit.
\end{definition}
The vacuum vector is actually an element of~$\V_0 \subseteq \V$ because it is invariant under the action of~$\C^\times$.
The case~$n=0$ of the associativity axiom implies that 
\[
  \mu(a_1,z_1,\ldots, a_m, z_m, \vac, z_{m+1}) = \mu(a_1,z_1,\ldots,a_m,z_m)
\]
for~$a_1,\ldots, a_m \in \V$ and~$z\in \C^{m+1} \setminus\Delta$.

The next proposition says that the action of~$\C^\times$ on~$\V$ extends to an action of the group~$\C^\times \ltimes \C$ of affine transformations of~$\C$
on the subspace
\begin{align}
  \Vbarbb = \{ x\in \Vbar \mid \exists C \in \Z\ \ \forall k < C : x_k = 0  \} \label{equation:Vbarbb}
\end{align}
of bounded below vectors.
More generally, when constructing geometric vertex algebras from vertex algebras, we will show that~$\mu(a_1,z_1,\ldots,a_m,z_m)$ vanishes in sufficiently low degree, i.\,e., is an element of~$\Vbarbb$.
If the~$\Z$-graded vector space~$\V$ is bounded from below, that is, if there is an integer~$K$ such that~$\V_k = 0$ for all~$k\leq K$, then~$\Vbarbb = \Vbar$. 
We do not need the next proposition to obtain a vertex algebra; rather it explains how the~$m=0$ case of associativity can be thought of as translation invariance of~$\mu$ by making translations act on~$\Vbarbb$ by using~$\mu$.  
\begin{proposition}\label{proposition:TranslationsActOnVbarbb}
  The vector space~$\Vbarbb$ is a representation of~$G= \C^\times \ltimes \C$ where~$w\in \C$ acts as~$w.x = \sum_{k\in \Z}\mu(p_k(x),w)$ and~$\lambda \in \C^\times$ acts as~$(\lambda.x)_k = \lambda^k x_k$ for~$x \in \Vbarbb$.
\end{proposition}
\begin{proof}
  For~$w_1,w_2 \in \C$ and~$x\in \Vbar_{bb}$
  \begin{align}
    w_1.(w_2.x) &= \sum_{k_1 \in \Z} \mu\left(p_{k_1}\left(\sum_{k_2 \in \Z} \mu(p_{k_2}(x),w_2)\right),w_1\right) \notag\\
    &= \sum_{k_1 \in \Z} \sum_{k_2 \in \Z} \mu(p_{k_1}(\mu(p_{k_2}(x),w_2)),w_1)\notag\\
    &= \sum_{k_2 \in \Z} \sum_{k_1 \in \Z} \mu(p_{k_1}(\mu(p_{k_2}(x),w_2)),w_1)\notag\\
    &= \sum_{k_2 \in \Z} \mu(p_{k_2}(x),w_2+w_1)\label{display:prop:TranslationsActOnVbarbb:Associativity}\\
    &= (w_1 +w_2). x\displaycomma \notag
  \end{align}
  where the sums are exchangeable because they are finite in each component of~$\Vbar$ and line~(\ref{display:prop:TranslationsActOnVbarbb:Associativity}) uses the associativity of~$\V$ for~$m=0$.
  Furthermore
  \begin{align*}
    0.x = \sum_{k\in \Z} \mu(p_k(x),0) = \sum_{k\in \Z} p_k(x) = x.    
  \end{align*}
  These actions of~$\C^\times$ and~$\C$ assemble to an action of~$\C^\times \ltimes \C$ because
  \begin{align*}
    \lambda. (w. (\lambda^{-1}.x)) &= \lambda . \sum_{k\in \Z}\mu(p_k(\lambda . x),w) \\ 
  &= \sum_{k\in \Z} \lambda . \mu(\lambda^{-1} . p_k(x),w) \\
  &= \sum_{k \in \Z} \mu(p_k(x),\lambda w)\\
  &= (\lambda w) . x
  \end{align*} 
  because of~$\C^\times$-equivariance. 
\end{proof}
\begin{definition}
If~$\V$ is a~$\Z$-graded vector space and~$v \in \V$ is homogeneous, then~$|v|$ denotes the degree of~$v$.
If~$f:\V \rightarrow \V'$ is a morphism of~$\Z$-graded vector spaces, then~$f$ extends to a linear map~$f:\Vbar \rightarrow \overline{\V'}$ given by~$f$ in each component.
\end{definition}
\begin{definition}
A \emph{morphism~$f:(\V, \mu) \rightarrow (\V', \mu')$ of geometric vertex algebras} is a morphism~$f:\V \rightarrow \V'$ of~$\Z$-graded vector spaces such that
\[
f(\mu(a_1,z_1,\ldots,a_n)) = \mu'(f(a_1),z_1,\ldots, f(a_n),z_n)
\]
for all~$a_1,\ldots,a_n \in \V$ and~$n \geq 0$.
\end{definition}
The definition of a~$\Z$-graded vertex algebra given in the next section involves the modes which are linear maps~$\V \rightarrow \V$.
We now define the modes in terms of a geometric vertex algebra, and show that they arise as the coefficients of the Laurent expansion of the function~$\mu(a,z,\blank,0)$ of~$z\in \C \setminus \{0\}$.
\begin{proposition}\label{proposition:ModeDefinitionAndDegree}
Let~$(\V,\mu)$ be a geometric vertex algebra. Let~$a\in \V$ and~$k \in \Z$. 
The~\emph{$k$-th mode} of~$a$ is a well-defined linear map
\labelledMapMapsto{a_{(k)}}{\V}{\V}{b}{a_{(k)}b := \ootpi \int_{S^1} z^k \mu(a,z;b,0) dz.}
If~$a\in \V$ is homogeneous, then~$a_{(k)}$ is homogeneous of degree~$|a| -k -1$.
A priori, the map~$a_{(k)}$ is well-defined as a map~$\V \rightarrow \Vbar$. 
Recall that we identify~$\V = \bigoplus_k \V_k$ with its image in~$\Vbar= \prod_k\V_k$ under the natural injection.
\end{proposition}
\begin{proof}
Let~$a,b\in \V$ be homogeneous.
Note that~$z.(a_{(k)}b) = z^{|a|-k-1 + |b|}a_{(k)}b$ for all~$z\in \C^\times$ because \begin{align*}
2\pi i \ z.(a_{(k)}b) &= z.\int_{S^1} w^k \mu(a,w;b,0) dw \\
&= \int_{S^1} w^k z.\mu(a,w;b,0) dw \\
&= \int_{S^1} w^k \mu(z.a,zw;z.b,0) dw \\ 
&= z^{|a|+|b|}\int_{S^1} w^k \mu(a,zw;b,0) dw \\ 
&= z^{|a|-k-1+|b|}\int_{zS^1} w^k \mu(a,w;b,0) dw &\text{ (substituted~$w/z$)}\\
&= z^{|a|-k-1+|b|}\int_{S^1} w^k \mu(a,w;b,0) dw &\text{ (holomorphic)}\\
&= 2\pi i\ z^{|a|-k-1+|b|} a_{(k)}b.
\end{align*}
This implies that~$p_l(a_{(k)} b)$ is zero if~$l\neq |a|-k-1+|b|$.
Thus~$a_{(k)}b \in \V_{|a| -k -1+|b|} \subseteq \V$.
Since every element of~$\V$ is a finite sum of homogeneous elements, it follows that~$a_{(k)}b \in \V$ for all~$a,b \in \V$.
\end{proof}
\begin{definition}
Let~$\V$ be a~$\Z$-graded vector space over~$\C$. 
Let~$A\subseteq \C$ be an annulus with center~$0$.
If~$f:A \rightarrow \Vbar$ is a holomorphic function, meaning in particular that it takes values in a finite-dimensional subspace in each degree, then we define the \emph{Laurent expansion}~$L(f)(x)\in \prod_{k\in \Z} \left(\V_k[[x^{\pm 1}]]\right)$ 
of~$f$ on~$A$ componentwise. 
\end{definition}
For~$\V$ a~$\Z$-graded vector space, we identify~$\V[[x^{\pm 1}]]$ with a subspace of
\[
\prod_{k\in \Z} \left(\V_k [[x^{\pm 1}]]\right) = \left(\prod_{k\in \Z} \V_k\right) [[x^{\pm 1}]] = \Vbar [[x^{\pm 1}]]
\]
via the injective linear map
\[
\sum_l A_l x^l \longmapsto \left(\sum_l p_k (A_l) x^l\right)_{k \in \Z}.
\]
\begin{proposition}
Let~$(\V, \mu)$ be a geometric vertex algebra.
For all~$a,b\in \V$, the Laurent expansion~$L[z \mapsto \mu(a,z,b,0)](x)$ is an element of~$\V[[x^{\pm 1}]]$ and is equal to~$\sum_l a_{(l)} b\ x^{-l-1}$.
\end{proposition}
\begin{proof}
  By the integral formula for the coefficients of the Laurent expansion 
  \begin{align*}
    L[z \mapsto \mu(a,z,b,0)](x)_k &= \sum_{l\in \Z}\ootpi \int_{S^1} z^lp_k(\mu(a,z,b,0)) dz \cdot x^{-l-1} \\
                                   &= \sum_{l\in \Z}  p_k\left(\ootpi \int_{S^1}z^l\mu(a,z,b,0)dz\right)x^{-l-1} \\
    &= \sum_{l\in \Z}  p_k(a_{(l)}b)x^{-l-1}. 
  \end{align*}
  For fixed~$l\in \Z$, the sum
  \(
    \sum_k p_k(a_{(l)}b)
  \)
  is finite and equals~$a_{(l)} b$, so~$\sum_l a_{(l)} b x^{-l-1}$ is the desired preimage.
\end{proof}
\begin{proposition}\label{proposition:MeromorphicIfBoundedBelow}
Assume that the underlying datum~$(\V,\mu)$ of a geometric vertex algebra satisfies all the axioms of a geometric vertex algebra except meromorphicity. If~$\V$ is bounded from below, then~$(\V,\mu)$ is meromorphic.
\end{proposition}
\begin{proof}
  If~$\V$ is bounded from below, then meromorphicity follows because
  \[
  \mu(b,w,c,0) = \sum_{k} b_{(k)}c w^{-k-1}
  \]
  for~$w\neq 0$, and~$|b_{(k)} c| = |b| + |c| - k - 1$ for~$b$ and~$c$ homogeneous, so~$b_{(k)} c$ is zero for~$k$ large enough.     
\end{proof}
\section{Vertex Algebras}
\begin{definition}
  A~$\Z$-graded \emph{vertex algebra} consists of:
  \begin{itemize}
  \item a~$\Z$-graded vector space~$\V = \bigoplus_{l\in \Z} \V_l$ over~$\C$, 
  \item a linear map
\labelledMapMapsto{Y}{\V}{\End{\V}[[x^{\pm 1}]]}{a}{ Y(a,x) = \sum_{k\in \Z} a_{(k)} x^{-k-1},}
where the~\emph{$k$-th mode}~$a_{(k)}$ is a homogeneous endomorphism of~$\V$ of degree~$|a| - k - 1$ for homogeneous~$a\in \V$,
\item a degree~$1$ endomorphism~$T$ of~$\V$, 
  \item a vector~$\vac \in \V_0$ called the \emph{vacuum},
\end{itemize}
such that:
\begin{itemize}
  \item (locality) For all~$a,b \in \V$, there exists a natural number~$N$ such that
    \[(x-y)^N[Y(a,x),Y(b,y)] = 0\]
    in~$\End(\V)[[x^{\pm 1}, y^{\pm 1}]]$.
  \item (translation) $T\vac = 0$ and~$[T,Y(a,x)] = \partial_x Y(a,x)$ for all~$a\in \V$.
  \item (creation) \( Y(a,x)\vac \in a + x \V[[x]] \) for all~$a\in \V$.
  \item (vacuum) \( Y(\vac,x) = \id_{\V} \).
\end{itemize}
\end{definition}
\begin{remark}
In terms of the modes of~$a \in \V$, the translation axiom for~$a\in \V$ is equivalent to demanding that
\(
  [T,a_{(k)}] = -k a_{(k-1)} 
\)
for all~$k \in \Z$.
The creation axiom for~$a$ is equivalent to the equation~$a_{(-1)}\vac = a$ and~$a_{(k)}\vac = 0$ for~$k\geq 0$.
The vacuum axiom says that for all~$a\in \V$ we have~$\vac_{(k)} a = 0$ for~$k \neq -1$ and~$\vac_{(-1)} a = a$.
\end{remark}
In a vertex algebra, the translation operator~$T$ and the vacuum~$\vac$ are uniquely determined by the vertex operators~$Y(a,x)$ as in the following proposition.
\begin{proposition}\label{proposition:VAStructureDeterminedByY}
  Let~$(\V,Y,T,\vac)$ be a vertex algebra.
  Then
  \(
    Ta = a_{(-2)}\vac 
  \)
  for all~$a\in \V$.
  If~$a\in \V$ satisfies~$Y(a,x) = \id_\V$, then~$a = \vac$.
\end{proposition}
\begin{proof}
  Let~$a\in \V$.
  It follows that
  \begin{align*}
    Ta = Ta_{(-1)}\vac = [T,a_{(-1)}]\vac + a T\vac = a_{(-2)}\vac + 0
  \end{align*}
  by the creation axiom and the translation axiom.
  If~$Y(a,x) = \id_\V$, then 
  \[
    a = Y(a,x)\vac|_{x=0} = \id(\vac)|_{x=0} = \vac
  \]
  by the creation axiom.
\end{proof}
\begin{definition}
A \emph{morphism~$f:(\V,Y,T,\vac) \rightarrow (\V', Y', T', \vac')$ of~$\Z$-graded vertex algebras} is a morphism~$f:\V \rightarrow \V'$ of~$\Z$-graded vector spaces such that
\begin{align}
f(Y(a,x)b) = Y'(f(a),x)f(b) \label{equation:MorphismOfVertexAlgebras}
\end{align}
for all~$a, b \in \V$ and $f(\vac) = \vac'$.
\end{definition}
\begin{remark}
Equation~\ref{equation:MorphismOfVertexAlgebras} says that~$f(a_{(k)} b) = f(a)_{(k)'} f(b)$ for all~$k\in \Z$.
It follows from Proposition~\ref{proposition:VAStructureDeterminedByY} that
\[
f(Ta) = f(a_{(-2)} \vac) = f(a)_{(-2)'} f(\vac) = f(a)_{(-2)'} \vac' = T'f(a)
\]
if~$f:(\V,Y,T,\vac) \rightarrow (\V', Y', T', \vac')$ is a morphism of vertex algebras.
\end{remark}

\section{The Equivalence of Categories}
In this section, we describe functors~$\Phi$ and~$\Psi$ relating the categories of geometric vertex algebras and~$\Z$-graded vertex algebras, and prove that they are well-defined on morphisms and inverse to each other on objects and morphisms.
Section~\ref{section:FromGeometricVertexAlgebrasToVertexAlgebras} shows that~$\Phi(\V,\mu)$ is a~$\Z$-graded vertex algebra if~$(\V,\mu)$ is a geometric vertex algebra.
Section~\ref{section:FromVertexAlgebrasToGeometricVertexAlgebras} shows that~$\Psi(\V,Y,T,\mu)$ is a geometric vertex algebra if~$(\V,Y,T,\mu)$ is a~$\Z$-graded vertex algebra.
\begin{definition}
  On morphisms, the functor~$\Phi$ from the category of geometric vertex algebras to the category of~$\Z$-graded vertex algebras is defined by the equation~$\Phi(f)=f$ of underlying morphisms of~$\Z$-graded vector spaces.
  
  On morphisms, the functor~$\Psi$ from the category of~$\Z$-graded vertex algebras to the category of geometric vertex algebras is defined by the equation~$\Psi(g) = g$ of underlying morphisms of~$\Z$-graded vector spaces.
\end{definition}
We now state a more detailed version of the theorem.
\begin{theorem}\label{theorem:GeometricVertexAlgebrasAndVertexAlgebrasDetailedVersion}
The category of geometric vertex algebras is isomorphic to the category of~$\Z$-graded vertex algebras.

  There is a bijection~$\Phi$ which maps a geometric vertex algebra~$(\V,\mu)$ to a~$\Z$-graded vertex algebra~$\Phi(\V,\mu)=(\V,Y_\mu,T_\mu,\vac_\mu)$ where
  \begin{align*}
    Y(a,x)b &= L[z\mapsto\mu(a,z,b,0)](x) \quad (\in \V[[x^{\pm 1}]]) \\
    Ta &= \partial_z \mu(a,z)|_{z=0}\\
    \vac &= \mu(\emptyset) 
  \end{align*}  
  for all~$a,b \in \V$.
  The inverse~$\Psi$ of~$\Phi$ is given by~$\Psi(\V,Y,T,\vac) = (\V,\mu_Y)$ where, for~$a=a_1,\ldots,a_m \in \V^{\otimes m}$, $m\geq0$, the function~$\mu_Y(a)$ is the unique~$\Vbar$-valued holomorphic function on~$\C^m\setminus\Delta$ such that
  \begin{align*}
    \mu_Y(a)(z) &= Y(a_1,z_1)\ldots Y(a_m,z_m)\vac
  \end{align*}
  for~$z=(z_1,\ldots,z_m)\in\C^m\setminus\Delta$ with~$|z_1| > \ldots > |z_m|$.

  If~$f:\V\rightarrow\V'$ be a morphism of~$\Z$-graded vector spaces, then~$f$ is a morphism
  $(\V,\mu)\rightarrow (\V',\mu')$ of geometric vertex algebras if and only if~$f$ is a morphism~$\Phi(\V,\mu) \rightarrow \Phi(\V',\mu')$ of~$\Z$-graded vertex algebras.
\end{theorem}
The right-hand side above is equal to
\begin{align*}
   &\sum_{k_1 \in \Z} {a_1}_{(k_1)}z^{-k_1-1}\ldots \sum_{k_m \in \Z} {a_m}_{(k_m)}z^{-k_m-1}\vac\\
   &=\sum_{k \in \Z^m} {a_1}_{(k_1)}\ldots {a_m}_{(k_m)}\vac z^{-k_1 -1}_1\ldots z^{-k_m -1}_m
\end{align*}
where every sum converges normally in each component of~$\Vbar$ by Proposition~\ref{proposition:ProductOfYs}.
Here, the modes~$a_{(k)}:\V\rightarrow \V$ for~$a\in \V,$~$k\in\Z$, and~$\V$ a~$\Z$-graded vertex algebra or geometric vertex algebra extend to linear maps~$a_{(k)} :\Vbar\rightarrow\Vbar$ as finite sums of homogeneous maps which all extend to linear maps~$\Vbar\rightarrow\Vbar$.
An equivalent definition is~$a_{(k)}b = \sum_{l\in \Z} a_{(k)} p_l b$ for~$a\in \V, b\in \Vbar$.
\begin{proposition}\label{proposition:PhiIsWellDefined}
  The functor~$\Phi$ is well-defined.
\end{proposition}
\begin{proof}
 Proposition~\ref{proposition:VertexAlgebraFromGeometricVertexAlgebra} says that~$\Phi(\V,\mu)$ is a well-defined~$\Z$-graded vertex algebra.
 If~$f:(V,\mu) \rightarrow (\V',\mu')$ is a morphism of geometric vertex algebras, then~$\Phi(f):(\V, Y_\mu, T_\mu, \vac_\mu)\rightarrow (\V',Y_{\mu'},T_{\mu'}, \vac_{\mu'})$ is a morphism of~$\Z$-graded vertex algebras because
\[
  \Phi(f)(\vac_\mu) = f(\vac_\mu) = f(\mu(\emptyset)) = \mu'(\emptyset) = \vac_{\mu'} = \vac_{\mu'}  
\]
and
\begin{align*}
  \Phi(f)(Y_\mu(a,x)b) &= f(Y_\mu(a,x)b)\\
  &= f\left( L[z\mapsto\mu(a,z,b,0)](x) \right)\\
  & = L[z\mapsto f(\mu(a,z,b,0))](x)\\
  &= L[z\mapsto \mu'(f(a),z,f(b),0)](x)\\
  & = Y_{\mu'}(f(a),x)f(b)\\
  & = Y_{\mu'}(\Phi(f)(a),x)f(b) 
\end{align*}
for all~$a,b \in \V$.
Since~$\Phi$ is the identity mapping on underlying morphisms of~$\Z$-graded vector spaces, it follows that~$\Phi$ is compatible with composition and preserves identity morphisms.
\end{proof}
\begin{proposition}\label{proposition:PsiIsWellDefined}
  The functor~$\Psi$ is well-defined.
\end{proposition}
\begin{proof}
Proposition~\ref{proposition:GeometricVertexAlgebraIsWellDefined} says that~$\Psi(\V,Y,T,\vac)$ satisfies the axioms of a geometric vertex algebra, and~$\mu_Y$ as in the statement of the Theorem is constructed in~Proposition~\ref{proposition:ProductOfYs}.
Let~$f:(\V,Y,T,\vac) \rightarrow (\V',Y',T',\vac')$ be a morphism of~$\Z$-graded vertex algebras.
If~$Y(a,z)b = \sum_{k\in\Z} a_{(k)}z^{-k-1}b$ converges absolutely in each component of~$\Vbar$ for~$a\in \V, b\in\Vbar, z\neq 0$, or all~$z\in\C$ if~$a_{(k)}b =0$ for~$k\geq 0$, then the sum defining~$Y(f(a),z)f(b)$ converges absolutely, too, and is equal to~$f(Y(a,z)b) \in \Vbar$ because
\begin{align*}
Y(f(a),z)f(b) &= \sum_{k\in\Z} f(a)_{(k)}z^{-k-1}f(b)\\
&= \sum_{k\in\Z} f(a_{(k)}z^{-k-1}b)\\
&= f\left( \sum_{k\in\Z} a_{(k)}z^{-k-1}b.\right)
\end{align*}
We conclude that~$\Psi(f):(\V,\mu_Y) \rightarrow (\V',\mu_{Y'})$ is a morphism of geometric vertex algebras because
\begin{align*}
\Psi(f)(\mu_Y(a_1,z_1,\ldots, a_m,z_m)) &= f(\mu_Y(a_1,z_1,\ldots, a_m,z_m))\\
&= f(Y(a_1,z_1)\ldots Y(a_m,z_m)\vac)\\
&= Y'(f(a_1),z_1)\ldots Y'(f(a_m),z_m)f(\vac)\\
&= Y'(f(a_1),z_1)\ldots Y'(f(a_m),z_m)\vac'\\
&= \mu_{Y'}(f(a_1),z_1,\ldots, f(a_m),z_m)\\
&= \mu_{Y'}(\Phi(f)(a_1),z_1,\ldots, \Phi(f)(a_m),z_m)
\end{align*}
for~$a_1,\ldots,a_m \in \V$ and~$z\in \C^m$ with~$|z_1| > \ldots > |z_m|$, and hence all~$z\in \C^m\setminus\Delta$ by the identity theorem.
Again, it follows that~$\Psi$ is compatible with composition and preserves identity morphisms.
\end{proof}
\begin{proof}[Proof of Theorem~\ref{theorem:GeometricVertexAlgebrasAndVertexAlgebrasDetailedVersion}]
  First, we prove that~$\Phi\after\Psi = \id$. Let~$(\V,Y,T,\vac)$ be a~$\Z$-graded vertex algebra.
  To prove~\(\Phi(\Psi(\V,Y,T,\vac)) = (\V,Y,T,\vac)\), it suffices to prove~$Y_{\mu_Y} = Y$ as a consequence of Proposition~\ref{proposition:VAStructureDeterminedByY}. By construction
  \begin{align*}
  Y_{\mu_Y}(a,x)b &= L[z\mapsto\mu_Y(a,z,b,0)](x)= L[z\mapsto Y(a,z)b](x) = Y(a,x)b
  \end{align*}
  for all~$a,b\in \V$.

  We now prove that~$\Psi\after\Phi = \id$.
  Let~$(\V,\mu)$ be a geometric vertex algebra.
  Note that~$\Psi(\Phi(\V,\mu)) = (\V,\mu_{Y_\mu})$ where
  \begin{align}
    \mu_{Y_\mu}(a_1,z_1,\ldots,a_m,z_m) &= Y_\mu(a_1,z_1)\ldots Y_\mu(a_m,z_m)\vac_\mu\notag\\
    &=\sum_{k_1\in \Z} {a_1}_{(k_1)}z^{-k_1-1}_1\sum_{k_m\in \Z}{a_m}_{(k_m)}z^{-k_m-1}_m\vac_\mu\label{equation:MuYMuInTermsOfModes}
  \end{align}
  for all~$a_1,\ldots, a_m \in \V$ and~$z\in \C^m\setminus\Delta$ with~$|z_1| > \ldots > |z_m|$ by Proposition~\ref{proposition:ProductOfYs} applied to the~$\Z$-graded vertex algebra~$\Phi(\V,\mu)=(\V,Y_\mu,T_\mu,\vac_\mu)$.
  The modes in~(\ref{equation:MuYMuInTermsOfModes}) are those of~$(\V,\mu)$, so Proposition~\ref{proposition:MuInTermsOfModes} implies that
  \[
    \mu_{Y_\mu}(a_1,z_1,\ldots,a_m,z_m) = \mu(a_1,z_1,\ldots,a_m,z_m)
  \]
  for all~$a_1,\ldots, a_m \in \V$ and~$z\in \C^m \setminus\Delta$ with~$|z_1| > \ldots > |z_m|$.
  Thus, for fixed~$a$, the functions~$\mu_{Y_\mu}(a,z)$ and~$\mu(a,z)$ of~$z \in \C^m\setminus\Delta$ are analytic continuations of the same function, and therefore agree by the identity theorem.

  The statement about morphisms in the theorem follows from Proposition~\ref{proposition:PhiIsWellDefined} and Proposition~\ref{proposition:PsiIsWellDefined}.
  Both~$\Phi$ and~$\Psi$ are the identity mapping on underlying morphisms of~$\Z$-graded vector spaces, so it follows that~$\Phi\after\Psi = \id$ and~$\Psi\after\Phi = \id$ on morphisms, too, proving that~$\Phi$ is an isomorphism with inverse~$\Psi$.
\end{proof}
\section{From Geometric Vertex Algebras to Vertex Algebras}\label{section:FromGeometricVertexAlgebrasToVertexAlgebras}
Given a geometric vertex algebra, we construct a~$\Z$-graded vertex algebra with the same underlying~$\Z$-graded vector space. 
\begin{proposition}\label{proposition:ModesInGeometricVertexAlgebra}
  Let~$\V$ be a geometric vertex algebra.
  Let~$a_1,\ldots,a_m \in \V$ and~$i,j \in \{1,\ldots,m\}$ with~$i<j$.
  For all~$(z_1,\ldots,z_{i-1},z_{i+1},\ldots, z_m) \in \C^{m-1}\setminus\Delta$ and~$\varepsilon > 0$ with~$\varepsilon < |z_l - z_j|$ for all~$l \neq i,j$,
  \begin{align*}
    &\ootpi\int_{\partial B_{\varepsilon}(z_j)} (z_i-z_j)^k\mu(a,z) dz_i \\
    &= \mu(a_1,z_1,\ldots, \widehat{a_i,z_i,}\ldots, a_{j-1},z_{j-1}, {a_i}_{(k)}a_j,z_j, a_{j+1}, z_{j+1}, \ldots, a_m,z_m).
  \end{align*}
\end{proposition}
\begin{proof}
  Using permutation invariance and associativity
  \begin{align*}
  \mu(a,z) = \sum_{l\in \Z}\mu(\ldots,\widehat{a_i,z_i,}\ldots,p_l\mu(a_i,z_i -z_j,a_j,0),z_j,\ldots)
  \end{align*}
  and this series of functions of~$z_i$ converges normally in each component of~$\Vbar$.
  It follows that we can exchange integration over~$\partial B_{\varepsilon}(z_j)$ with summation to get
  \begin{align*}
    &\ootpi\int_{\partial B_{\varepsilon}(z_j)} (z_i-z_j)^k\mu(a,z) dz_i \\
    &= \sum_{l\in \Z}\mu(\ldots,\widehat{a_i,z_i,}\ldots,p_l\ootpi\int_{\partial B_{\varepsilon}(z_j)} (z_i - z_j)^k\mu(a_i,z_i -z_j,a_j,0)dz_i,z_j,\ldots)dz_i. 
  \end{align*}
  Here, we may move the integral into the argument of~$\mu$ and~$p_l$ because these maps are linear and the relevant functions take values in finite-dimensional subspaces.
  Shifting the contour integral to zero and noting that it does not matter whether we integrate around a circle of radius~$\varepsilon$ or~$1$ as in the definition of the modes, we get
  \begin{align*}
    \ootpi\int_{\partial B_{\varepsilon}(z_j)} (z_i - z_j)^k\mu(a_i,z_i -z_j,a_j,0)dz_i &= \ootpi\int_{\partial B_{\varepsilon}(0)} w^k\mu(a_i,w,a_j,0)dw \\
    &= {a_i}_{(k)} a_j
  \end{align*}
  and thus
  \begin{align*}
    &\ootpi\int_{\partial B_{\varepsilon}(z_j)} (z_i-z_j)^k\mu(a,z) dz_i = \sum_{l\in \Z}\mu(\ldots,\widehat{a_i,z_i,}\ldots,p_l{a_i}_{(k)} a_j,z_j,\ldots) \\
    &= \mu(\ldots,\widehat{a_i,z_i,}\ldots,\sum_{l \in \Z}p_l{a_i}_{(k)} a_j,z_j,\ldots) = \mu(\ldots,\widehat{a_i,z_i,}\ldots,p_l{a_i}_{(k)} a_j,z_j,\ldots)
  \end{align*}
  because the sum is finite.
\end{proof}
For~$i,j \in \{1,\ldots,m\}$ with~$i < j$, 
let~$U_{ij}$ be the set of~$z \in \C^m \setminus \Delta$ with~$|z_i - z_j| < \min_{l \neq i,j} |z_l - z_j|$.
The following proposition is the analogue of Proposition~5.3.6 from~\cite[p. 167]{CostelloGwilliam} for geometric vertex algebras.
It expresses products for~$z_i$ close to~$z_j$ in terms of a series in~$z_i -z_j$ and other vertex algebra elements inserted at~$z_j$ and goes by the name of \emph{operator product expansion (OPE)}.
\begin{proposition}\label{proposition:OPEforGeometricVertexAlgebras}
  Let~$\V$ be a geometric vertex algebra and let~$a_1,\ldots,a_m \in \V$.
  For~$z \in U_{ij}$
  \[
    \mu(a_1,z_1,\ldots, a_m,z_m) = \sum_{k \in \Z} \mu(\ldots, \widehat{a_i,z_i,} \ldots, {a_i}_{(k)}a_j,z_j, \ldots)(z_i - z_j)^{-k-1}
  \]
  with normal convergence. 
\end{proposition}
\begin{proof}
  For fixed~$z_1,\ldots,z_{i-1},z_{i+1},\ldots,z_m$ the left-hand side is a holomorphic function of~$z_i$ such that~$z\in U_{ij}$.
  The integral formula for the coefficients of the normally convergent Laurent expansion and Proposition~\ref{proposition:ModesInGeometricVertexAlgebra} imply
  \begin{align*}
    \mu(a,z) &= \sum_{k\in \Z} \ootpi \int_{\partial B_\varepsilon(z_j)} (w-z_j)^k\mu(\ldots, a_i,w, \ldots,a_j,z_j, \ldots) dw (z_i - z_j)^{-k-1} \\
    &= \sum_{k\in \Z} \mu(\ldots,\widehat{a_i,z_i,}\ldots,{a_i}_{(k)}a_j,z_j,\ldots)(z_i - z_j)^{-k-1}.
  \end{align*}
\end{proof}
\begin{proposition}\label{proposition:VertexAlgebraFromGeometricVertexAlgebra}
  If~$(\V,\mu)$ is a geometric vertex algebra, then~$(\V,Y_\mu,T_\mu,\vac_\mu)$ is a~$\Z$-graded vertex algebra.
\end{proposition}
\begin{proof}
  We determine the degrees of~$T_\mu$ and~$\vac_\mu$  using the action of~$\C^\times$ and the equivariance axiom similarly to how we determined the degrees of the modes in Proposition~\ref{proposition:ModeDefinitionAndDegree}:
  For all~$z\in \C^\times$,
  \[
    z.\vac_\mu = z. \mu(\emptyset) = \mu(\emptyset) = \vac_\mu
  \]
  so~$\vac\in \V_0 \subseteq \Vbar$ since this shows that~$p_l\vac = 0$ for~$l\neq 0$.
  For all~$z\in \C^\times$ and~$a\in \V$ homogeneous, we have
  \begin{align*}
    z.T_\mu a &= z.\partial_w\mu(a,w)|_{w=0} = \partial_wz.\mu(a,w)|_{w=0} = \partial_w \mu(z.a,zw)|_{w=0} \\ 
    &= \partial_w \mu(z^{|a|}a,zw)|_{w=0} = z^{|a|+1}\partial_w\mu(a,w)|_{w=0} = z^{|a|+1} T_\mu a
  \end{align*} 
  so~$T_\mu a$ is concentrated in degree~$|a|+1$.
  It follows that the image of~$T_\mu$ is a subset of~$\V$, and~$T_\mu$ has degree~1.

  \paragraph{Locality:}
  Let~$a,b,c \in \V$.
  It suffices to treat the case of~$a,b,c$ homogeneous.
  Applying Proposition~\ref{proposition:OPEforGeometricVertexAlgebras} twice, we find
  \begin{align*}
    \mu(a,z,b,w,c,0) &= \sum_l \mu(a,z,b_{(l)}c\ w^{-l-1},0) = \sum_l \sum_k a_{(k)}b_{(l)}c\ z^{-k-1}w^{-l-1} 
  \end{align*}
  with normal convergence for~$|z| > 1 > |w| > 0$.
  Let
\[
A_{r,R}= \{ z \in \C \mid r < |z| < R \}
\]
for~$0 \leq r,R \leq \infty$.
  Using a similar notation for the Laurent expansion of functions of several variables on products of annuli, we get
  \begin{align*}
    L_{(z,w) \in A_{1,2} \times A_{0,1}} \mu(a,z,b,w,c,0) (x,y) = Y_\mu(a,x)Y_\mu(b,y)c
  \end{align*}
  and analogously
  \begin{align*}
   L_{(z,w) \in A_{0,1} \times A_{1,2}} \mu(a,z,b,w,c,0) (x,y) = Y_\mu(b,y)Y_\mu(a,x)c.
  \end{align*}
  For all~$N\in \N$
  \begin{align*}
    L_{(z,w) \in A_{1,2} \times A_{0,1}} (z-w)^ N \mu(a,z,b,w,c,0) (x,y) &= (x-y)^NY_\mu(a,x)Y_\mu(b,y)c \\
    L_{(z,w) \in A_{0,1} \times A_{1,2}} (z-w)^N \mu(a,z,b,w,c,0) (x,y) &= (x-y)^NY_\mu(b,y)Y_\mu(a,x)c
  \end{align*}
  because Laurent expansion intertwines the action of polynomials as functions with the action of polynomials as formal polynomials. 
  For~$N$ large enough, the function
  \(
    (z-w)^N\mu(a,z,b,w,c,0) 
  \)
  of~$(z,w) \in A_{0,2} \times A_{0,2} \setminus \Delta$ extends to a holomorphic function of~$(z,w) \in A_{0,2}\times A_{0,2}$ as a consequence of Proposition~\ref{proposition:OPEforGeometricVertexAlgebras} about the OPE and meromorphicity. 
  If~$F \in \mathcal{O}(A_{0,2} \times A_{0,2})$, then~$L_{A_{1,2} \times A_{0,1}} F = L_{A_{0,1} \times A_{1,2}}F$ in~$\C[[x^{\pm 1}, y^{\pm 1}]]$, 
  and therefore
  \begin{align*}
    (x-y)^N Y_\mu(a,x)Y_\mu(b,y)c &= (x-y)^N Y_\mu(b,y)Y_\mu(a,x)c
  \end{align*}
  in~$\V[[x^{\pm 1}, y^{\pm 1}]]$. 
  \paragraph{Vacuum:}
  Since
  \[
    \mu(\vac_\mu,z,a,0) = \mu(\mu(\emptyset),z,a,0) = \mu(a,0) = a
  \]
  by the definition of~$\vac_\mu$, associativity, and insertion at zero, we have that
  \begin{align*}
    {\vac_\mu}_{(k)} a= \ootpi \int_{S^1} z^k \mu(\vac_\mu,z,a,0) dz = \ootpi\int_{S^1} z^k a dz = \delta_{k,-1} a   
  \end{align*}
  for all~$a\in \V$ and~$k\in \Z$.
  \paragraph{Creation:}
  Let~$a\in \V$.
  By associativity
  \(
    \mu(a,z,\vac_\mu,0) = \mu(a,z)
  \)
  and this is a holomorphic function of~$z\in \C$.
  Therefore, its Laurent expansion has no negative powers and its zeroth term is~$\mu(a,0) = a$ because of the axiom about insertion at zero.
  \paragraph{Translation:}
  Let~$T_\mu:\Vbar \rightarrow \Vbar$ be defined by~$T_\mu x = \sum_k T_\mu p_kx$. This sum is finite in each degree, and~$T_\mu$ is linear.
  Let~$a,b\in \V$.
  The identity
  \[
    [T_\mu,Y_\mu(a,x)]b = T_\mu Y_\mu(a,x)b - Y_\mu(a,x)T_\mu b = \partial_xY_\mu(a,x)b.
  \] 
  is implied by 
  \begin{align}
    T_\mu \mu(a,z,b,0) - \mu(a,z,T_\mu b,0) = \partial_z\mu(a,z,b,0) \label{display:TranslationSubgoal}
  \end{align}
  because Laurent expansion is compatible with linear maps and with differentiation.
  Recall that~$T_\mu c = \partial_w  \mu(c,w)|_{w=0}$ for~$c\in \V$.  
  We may assume that~$w$ is close to zero to compute the~$w$-derivative. 
  \begin{align}
    \mu(a,z,T_\mu b,0) &=  \mu(a,z,\sum_{k}p_kT_\mu b,0) \label{display:muT1}\\ 
                  &= \sum_k\mu(a,z,p_k\partial_w \mu(b,w)|_{w=0},0)\label{display:muT2}\\ 
                  &= \sum_k \partial_w\mu(a,z,p_k\mu(b,w),0)|_{w=0} \label{display:muT3} \\
                  &= \left[\partial_w\sum_k \mu(a,z,p_k\mu(b,w),0)\right]_{w=0} \label{display:muT4}\\
                  &= \partial_w \mu(a,z,b,w)|_{w=0} \label{display:muT5}
  \end{align}
Equation~(\ref{display:muT5}) follows from associativity which implies that the sum in~(\ref{display:muT4}) is normally convergent, and this implies that we can commute the sum and the derivative in Equation~(\ref{display:muT4}), so the sum in~(\ref{display:muT3}) converges normally.
Equation~(\ref{display:muT3}) uses~$p_k$ is linear and that~$\mu$ is linear in each argument from~$\V$. 
The sums in~(\ref{display:muT1}) and~(\ref{display:muT2}) are finite. 
Similarly,
  \begin{align*}
    T_\mu \mu(a,z,b,0) &=  \sum_k\partial_w\mu(p_k\mu(a,z,b,0),w)|_{w=0}\\ 
                  &= \partial_w\sum_k\mu(p_k\mu(a,z,b,0),w)\big|_{w=0}\\  
                  &= \partial_w\mu(a,z+w,b,w)|_{w=0} 
  \end{align*}
  and thus
  \begin{align*}
    T_\mu \mu(a,z,b,0) - \mu(a,z,T_\mu b,0) &= \partial_w\mu(a,z+w,b,w)|_{w=0} - \partial_w \mu(a,z,b,w)|_{w=0}\\
                                  &= \partial_w\mu(a,z+w, b,0)|_{w=0}\\
                                  &= \partial_z\mu(a,z,b,0),
  \end{align*}
  which is Equation~(\ref{display:TranslationSubgoal}).
\end{proof}
The following proposition expresses~$\mu(a,z)$ in terms of the modes.
\begin{proposition}\label{proposition:MuInTermsOfModes}
  If~$(\V,\mu)$ is a geometric vertex algebra, then
  \begin{align*}
   \mu(a_1,z_1,\ldots,a_m,z_m) = \sum_{k\in\Z^m}{a_1}_{(k_1)}\ldots{a_m}_{(k_m)}\vac_\mu z^{-k_1-1}_1\ldots z^{-k_m-1}_m
  \end{align*}
  for~$a_1,\ldots,a_m \in \V$ and~$z\in \C^m \setminus\Delta$ with~$|z_1| > \ldots > |z_m|$.
\end{proposition}
\begin{proof}
  For~$m=0$, this says that~$\mu(\emptyset) = \vac_\mu$ which is the definition of~$\vac_\mu$.
We repeatedly apply Proposition~\ref{proposition:OPEforGeometricVertexAlgebras} to deduce that
  \begin{align*}
    \mu(a,z) &= \mu(a_1,z_1,\ldots,a_m,z_m, \vac_\mu, 0)\\
  &=\sum_{k_m\in\Z}\ldots\sum_{k_1\in\Z} {a_1}_{(k_1)}\ldots{a_m}_{(k_m)}\vac_\mu z^{-k_1-1}_1\ldots z^{-k_m-1}_m  
  \end{align*}
for all~$a_1,\ldots, a_m \in \V$ and~$z\in \C^m \setminus\Delta$ with~$|z_1| > \ldots > |z_m|$. This holomorphic function of such~$z$ has the same Laurent expansion as the right-hand side in the claim.
\end{proof}
\section{From Vertex Algebras to Geometric Vertex Algebras}\label{section:FromVertexAlgebrasToGeometricVertexAlgebras}
\begin{proposition}\label{proposition:VanishingOfModes}
  Let~$\V$ be a~$\Z$-graded vertex algebra. For all~$a,b \in \V$, there exists an integer~$N$ such that 
  \(
    a_{(n)} b = 0
  \)
  if~$n \geq N$.
\end{proposition}
This follows for degree reasons in the bounded below case.
The proof given below does not use this assumption and serves a warm-up for the construction of~$\mu_Y$.
\begin{proof}
  Let~$a,b\in \V$.
  By the locality axiom, there exists an~$N$ such that
  \begin{align}
  (x-y)^N Y(a,x)Y(b,y)\vac = (x-y)^N Y(b,y)Y(a,x)\vac\displayperiod \label{display:Field}
  \end{align}
  By the creation axiom for~$b$, we may set~$y=0$ on the left-hand side and get
  \[
  x^N Y(a,x) b = \sum_n a_{(n)} b x^{N-n-1}\displayperiod
  \]
  Using Equation~(\ref{display:Field}) and the creation axiom for~$a$ on the right-hand side, it follows that the left-hand side\ contains no negative powers of~$x$, even after setting~$y=0$.
  This means that~$a_{(n)} b = 0$ for~$N - n -1 < 0$, equivalently~$n \geq N$.
\end{proof}
\begin{proposition}\label{proposition:ProductOfYs}
  Let~$(\V,Y,T,\vac)$ be a~$\Z$-graded vertex algebra.
  Fix an integer~$m\geq 0$. For~$a_1,\ldots,a_m \in \V$ and~$z \in \C^m \setminus\Delta$ the sums
  \begin{align}
  Y(a_1,z_1)\ldots Y(a_m,z_m)\vac &:= \sum_{k_1 \in \Z} {a_1}_{(k_1)}z^{-k_1-1}_{1}\ldots \sum_{k_m \in \Z} {a_m}_{(k_m)}z^{-k_m-1}_{m} \vac\label{equation:ProductOfYs}\\
  &=\sum_{k \in \Z^m} {a_1}_{(k_1)}z^{-k_1-1}_{1}\ldots {a_m}_{(k_m)}z^{-k_m-1}_{m} \vac \label{equation:ProductOfYsSingleSum}
  \end{align}
  converge normally in each component of~$\Vbar$
  for~$z\in D_m$, where
  \[
    D_m := \{ z\in \C^m \mid |z_1| > \ldots > |z_m| \} \subseteq \C^m \setminus\Delta\displayperiod
  \]
  As a function of~$z$, the value of this series extends to a unique holomorphic $\Vbar$-valued function
  \[
    \mu_Y(a,z) = \mu_Y(a_1,z_1,\ldots, a_m,z_m)
  \]
  of~$z\in \C^m \setminus\Delta$. 
  It satisfies~$\mu_Y(a^\sigma, z^\sigma) = \mu_Y(a,z)$ for every permutation~$\sigma \in \Sigma_m$.
  Furthermore,~$\mu_Y(a)$ is identically zero in sufficiently low degree for every~$a \in \V^{\otimes m}$. 
\end{proposition}
This proposition and its proof are very similar to what is found in~\cite[1.2 and 4.5]{FrenkelBenZvi},~\cite[5.3]{HuangTwoDimConfGeomAndVOAs},~\cite[A.2]{FrenkelLepowskyMeurman},~\cite[3.5.1]{FrenkelHuangLepowsky}, which treat the bounded below, degreewise finite-dimensional case in these parts.
\begin{proof}
  We may assume that~$a_1,\ldots, a_m$ are homogeneous.
  Let
\[f(a,x) = \sum_{k \in \Z^m} {a_1}_{(k_1)} \ldots {a_m}_{(k_m)} \vac x^{-k_1-1}_{1}\ldots x^{-k_m-1}_{m} \in \V[[x^{\pm 1}_1,\ldots x^{\pm 1}_m]] \]
  denote the formal series corresponding to the sum in Equation~(\ref{equation:ProductOfYsSingleSum}), whose proof we defer until the end of this proof. 
  By the creation axiom for~$a_m$, there are no negative powers of~$x_m$ in~$f(a,x)$.
  The locality axiom for~$a_i$ and~$a_j$ implies that there is a natural number~$N_{ij}$ with
  \[
   (x_i - x_j)^{N_{ij}} [Y(a_i, x_j),Y(a_j,x_j)] = 0
  \]
  for~$1 \leq i < j \leq m$.
  Let
  \begin{align}
    g(x) = \prod_{i<j} (x_i - x_j)^{N_{ij}} \label{equation:ProductOfYsAbsolutelyConvergent:gDefinition}
  \end{align}
  and it follows that
  \begin{align}
    g(x)f(a^\sigma, x^\sigma) = g(x) f(a,x) \label{equation:ExistenceOfMuRequirementFor_g}
  \end{align}
  for all permutations~$\sigma \in \Sigma_m$. This can be checked by considering the special case of~$\sigma$ a transposition.
  Equation~\ref{equation:ExistenceOfMuRequirementFor_g} and the creation axiom for~$a_i$ imply that there are no negative powers of~$x_i$ in~$g(x)f(a,x)$ for any~$i=1,\ldots,m$.
Let~$l\in \Z$.
The~$\V_l$-component~$p_l(g(x)f(a,x))$ has no negative powers of~$x_i$ for any~$i=1,\ldots,n$.
We claim that~$p_l(g(x)f(a,x))$ is a polynomial in the~$x_i$ with coefficients in~$\V_l$.
Let~$r_\beta \in \Z$ be the coefficient of~$x^\beta$ in~$g(x)$.
The coefficient of~$x^\alpha$ in 
\begin{align}
 p_l(g(x)f(a,x)) = \sum_{\alpha \in \Z^m} \left( \sum_{\beta + k = \alpha} r_\beta p_l({a_1}_{(k_1)}\ldots {a_m}_{(k_m)}\vac) \right) x^\alpha \label{display:ConstructionOfMu:SumOfMonomials}
\end{align}
is zero if~$\alpha_i < 0$ for some~$i=1,\ldots,m$. 
Let~$\alpha \in \Z^m$ be such that the coefficient of~$x^\alpha$ is not zero. 
There is at least one pair~$(\beta,k)$ with~$\beta + k = \alpha$ and
\[
r_\beta p_l({a_1}_{(k_1)}\ldots {a_m}_{(k_m)}\vac) \neq 0\displayperiod
\]
This implies that~$r_\beta \neq 0$ and~$p_l({a_1}_{(k_1)}\ldots {a_m}_{(k_m)}\vac) \neq 0$.
Thus~$\beta$ is in a fixed finite subset of~$\Z^m$ depending only on the~$N_{ij}$, and~$k$ satisfies 
\begin{align}
  l = |{a_1}_{(k_1)}\ldots {a_m}_{(k_m)}\vac| = |a_1| - k_1 - 1 + \ldots + |a_m| - k_m - 1\displaycomma \label{equation:lAndDegreeOfABunchOfModes}
\end{align}
so
\begin{align*}
\sum^m_{i=1} \alpha_i = \sum^m_{i} (\beta_i + k_i) = \sum^m_{i} (\beta_i + |a_i|) - n - l.
\end{align*}
For fixed~$\beta$, the set of such~$\alpha$ is a hyperplane intersecting~$\Z^m_{\geq 0}$ in finitely many points. 
It follows that~$\alpha$ is in a finite subset of~$\Z^m$ depending only on the~$N_{ij}$.   
  In summary, each~$\V_l$-component of~$g(x)f(a,x)$ is a polynomial in the~$x_i$ with coefficients in~$\V_l$. 
  Thus
  \[F_a(z) = [g(x)f(a,x)]_{x=z}\] 
  defines a holomorphic function~$F_a:\C^m\rightarrow \Vbar$.
  Let~$A \subseteq D_m$ be a product of annuli centered at~$0$.
  Let
  \[
    L: \mathcal{O}(A;\Vbar) \longrightarrow \prod_k \V_k [[x^{\pm 1}_1,\ldots,x^{\pm 1}_m]]
  \]
  denote the Laurent expansion map on~$A$. 
  The Laurent expansion on~$A$ of holomorphic functions defined on~$D_m$ is independent of the choice of~$A$. 
  In particular, the Laurent expansion~$L(F_a)$ of~$F_a$ on~$A$ is independent of the choice of~$A$ and satisfies~$L(F_a)(x) = g(x)f(a,x)$.
  Then
  \[
    \mu_Y(a_1,z_1,\ldots,a_m,z_m) = g(z)^{-1} F_a(z) \quad\quad (z\in \C^m \setminus\Delta)
  \] defines a holomorphic function on~$\C^m \setminus\Delta$ with values in~$\Vbar$ whose Laurent expansion on~$A$, again independent of the choice of~$A$, is equal to~$f(a,x)$ because 
\begin{align*}
  &L[z \mapsto \mu_Y(a_1,z_1,\ldots, a_n,z_n)](x) = L(F_a/g)(x) = L(1/g)(x) L(F_a)(x) \\
  &= L(1/g)(x) g(x)f(a,x) =  L(1/g)(x) L(g)(x) f(a,x) = L(1)(x) f(a,x) = f(a,x).
\end{align*}
Thus~$f(a,z)$ converges normally for~$z \in A$ and~$f$ and~$\mu_Y$ agree on~$A$.
  For every~$z \in D_m$ there is an~$A$ as above with~$z\in A$, so~$f(a,z)$ converges normally for all~$z\in D_m$ and~$f$ and~$\mu_Y(a)$ agree on all of~$D_m$. 
  The uniqueness of~$\mu_Y(a)$ follows by the identity theorem for analytic functions because~$\C^m \setminus\Delta$ is connected,~$D_m \subseteq \C^m$ is open and non-empty, and~$\mu_Y(a,z) = f(a,z)$ for~$z\in D_m$. 
Note that the identity theorem for holomorphic functions with values in a finite-dimensional vector space applies here because there is a finite-dimensional subspace of~$\V_l$ for each~$l \in\Z$ to which the function~$z\mapsto p_l(\mu_Y(a,z))$ maps holomorphically. 

We now prove invariance under permutation. Let~$\sigma \in \Sigma_m$. We now write~$g_N$ for~$g$ to make the dependence of~$g$ on the~$N_{ij}$ explicit. Let~$\widetilde{N}$ denote a choice like~$N$ for~$a^\sigma$ instead of~$a$ (e.\,g.,~$\widetilde{N}_{ij} = N_{\sigma(i)\sigma(j)}$).
Our goal is the equality of
\begin{align*}
  \mu_Y(a,z) &= [g_N(x) f(a,x)]_{x=z} g_N(z)^{-1}\\
  \intertext{with}
  \mu_Y(a^\sigma,z^\sigma) &= [g_{\widetilde{N}}(x), f(a^\sigma,x)]_{x =z^\sigma} g_{\widetilde{N}}(z^\sigma)^{-1}\\
  &= [g_{\widetilde{N}}(x^\sigma), f(a^\sigma,x^\sigma)]_{x =z} g_{\widetilde{N}}(z^\sigma)^{-1}
\end{align*}
for~$z \in \C^m \setminus \Delta$. 
Starting with Equation~(\ref{equation:ExistenceOfMuRequirementFor_g}), it follows that
\begin{align*}
  g_N(x)f(a,x) &= g_N(x)f(a^\sigma,x^\sigma)\\
  g_{\widetilde{N}}(x^\sigma)g_N(x)f(a,x) &= g_N(x)g_{\widetilde{N}}(x^\sigma)f(a^\sigma,x^\sigma)\\
  [g_{\widetilde{N}}(x^\sigma)g_N(x)f(a,x)]_{x=z} &= [g_N(x)g_{\widetilde{N}}(x^\sigma)f(a^\sigma,x^\sigma)]_{x=z}\\
  g_{\widetilde{N}}(z^\sigma)[g_N(x)f(a,x)]_{x=z} &= g_N(z)[g_{\widetilde{N}}(x^\sigma)f(a^\sigma,x^\sigma)]_{x=z}\\
  [g_N(x)f(a,x)]_{x=z}g_N(z)^{-1} &= [g_{\widetilde{N}}(x^\sigma)f(a^\sigma,x^\sigma)]_{x=z}g_{\widetilde{N}}(z^\sigma)^{-1}\\
  \mu_Y(a,z)&= \mu_Y(a^\sigma,z^\sigma).
\end{align*}
We now prove that~$\mu_Y(a,z)$ is zero in sufficiently low degree. 
It suffices to prove that~$p_l(g(x)f(a,x))$ is zero for~$l$ low enough.
This formal series in the~$x_i$ with coefficients in~$\V_l$ has no negative powers of~$x_i$ for any~$i$.
Therefore, for a monomial in~(\ref{display:ConstructionOfMu:SumOfMonomials}) with non-zero coefficient, the~$-k_i-1$ in Equation~(\ref{equation:lAndDegreeOfABunchOfModes}) are bounded from below, so~$l$ is as well, and the bound is independent of~$z$.

It remains to prove Equation~(\ref{equation:ProductOfYsSingleSum}) and the normal convergence of the sums in Equation~(\ref{equation:ProductOfYs}).
For~$m=0$, there is nothing to show.
Assume that the claim holds for~$m-1$ with~$m\geq 1$. In particular, 
\begin{align*}
&\sum_{k_2,\ldots,k_m \in\Z} {a_2}_{(k_2)} z^{-k_2-1}_2 \ldots {a_m}_{(k_m)} z^{-k_m-1}_m \vac\\
&= \sum_{k_2\in\Z} {a_2}_{(k_2)} z^{-k_2-1}_2 \ldots\sum_{k_m\in\Z} {a_m}_{(k_m)} z^{-k_m-1}_m \vac
\end{align*}
for~$(z_2,\ldots,z_m) \in D_{m-1}$ with all sums normally convergent. Applying~${a_1}_{(k_1)}z^{-k_1-1}_1$ to both sides, which preserves normally convergent sums, we get
\begin{align*}
&\sum_{k_2,\ldots,k_m \in\Z} {a_1}_{(k_1)}z^{-k_1-1}_1 {a_2}_{(k_2)} z^{-k_2-1}_2 \ldots {a_m}_{(k_m)} z^{-k_m-1}_m \vac\\
&= {a_1}_{(k_1)}z^{-k_1-1}_1\sum_{k_2\in\Z} {a_2}_{(k_2)} z^{-k_2-1}_2 \ldots\sum_{k_m\in\Z} {a_m}_{(k_m)} z^{-k_m-1}_m \vac
\end{align*}
for~$z=(z_1,z_2,\ldots,z_m) \in D_m$. Normal convergence of the sum over~$k\in \Z^m$ defining~$f(a,z)$ for~$z \in D_m$ implies that the sums over~$k_1\in\Z$ in
\begin{align*}
&\sum_{k\in\Z^m}{a_1}_{(k_1)}z^{-k_1-1}_1 {a_2}_{(k_2)} z^{-k_2-1}_2 \ldots {a_m}_{(k_m)} z^{-k_m-1}_m \vac\\
&=\sum_{k_1\in\Z}\sum_{k_2,\ldots,k_m \in\Z} {a_1}_{(k_1)}z^{-k_1-1}_1 {a_2}_{(k_2)} z^{-k_2-1}_2 \ldots {a_m}_{(k_m)} z^{-k_m-1}_m \vac\\
&= \sum_{k_1\in\Z} {a_1}_{(k_1)}z^{-k_1-1}_1\sum_{k_2\in\Z} {a_2}_{(k_2)} z^{-k_2-1}_2 \ldots\sum_{k_m\in\Z} {a_m}_{(k_m)} z^{-k_m-1}_m \vac
\end{align*}
are normally convergent, too.
\end{proof}
\begin{proposition}\label{proposition:GeometricVertexAlgebraIsWellDefined}
The pair $(\V,\mu_Y)$ satisfies the axioms of a geometric vertex algebra. 
\end{proposition}
\begin{proof}
  Let~$m \geq 0$.
  Both the multilinearity and~$\C^\times$-equivariance may be checked on~$D_m$ by the uniqueness of analytic continuations from~$D_m$ to~$\C^m \setminus\Delta$. 
  On~$D_m$, we can express~$\mu_Y$ in terms of~$Y$ and multilinearity follows from the fact that~$Y$ is linear and composition in~$\End(\V)$ is bilinear.
  For~$\C^\times$-equivariance, it suffices to consider homogeneous~$a_1,\ldots, a_m \in \V$:
  for all~$z \in \C^\times$ and~$w \in D_m$,
  \begin{align*}
    &z.\mu_Y(a,w) \\
    &= \sum_{k_1,\ldots,k_m \in \Z} z.\left({a_1}_{(k_1)}w^{-k_1-1}_1 \ldots {a_m}_{(k_m)} w^{-k_m - 1}_m\vac\right)  \\
               &= \sum_{k_1,\ldots,k_m \in \Z} z^{|a_1| - k_1 - 1 + \ldots+|a_m| - k_m - 1} {a_1}_{(k_1)}w^{-k_1-1}_1 \ldots {a_m}_{(k_m)} w^{-k_m - 1}_m\vac  \\
    &= \sum_{k_1,\ldots,k_m \in \Z} (z.a_1)_{(k_1)}(zw_1)^{-k_1-1} \ldots (z.a_m)_{(k_m)} (zw_m)^{-k_m - 1}\vac \\
    &= \mu_Y(z.a_1,zw_1,\ldots, z.a_m,zw_m).
  \end{align*}
  The axiom about insertion at zero follows from the creation axiom of the vertex algebra: for all~$a\in \V$,
  \begin{align*}
    \mu_Y(a,0) = Y(a,0)\vac = a. 
  \end{align*}
  To see that the meromorphicity axiom is satisfied, let~$N$ be the maximum of the~$N_{ij}$ from the construction of~$\mu_Y$, after fixing some elements of~$\V$.
  By construction,~$\mu_Y$ satisfies the following strong version of meromorphicity:
\begin{itemize}
\item For all~$a_1,\ldots, a_n\in \V$, there is a natural number~$N$ such that the function
   \[
     z  \longmapsto \prod_{i<j} (z_i - z_j)^N \mu(a_1,z_1,\ldots, a_n,z_n)
   \]
   with values in~$\Vbar$ has a holomorphic extension from~$\C^n \setminus\Delta$ to~$\C^n$.
\end{itemize}
Proposition~\ref{proposition:MuIsAssociative}, proven further below, says that~$\mu_Y$ satisfies the associativity axiom.
\end{proof}
The next proposition establishes how~$\mu_Y$ behaves with respect to translations.
It is used in the proof of associativity to arrange~$z_{m+1} = 0$.
A translation by~$z \in \C$ acts on~$\Vbarbb$ (defined by Equation~(\ref{equation:Vbarbb})) via the automorphism~$e^{zT}$ of~$\Vbarbb$ defined by
\[
  e^{zT} x = \sum_{k \in \Z} e^{zT} p_k(x)
\]
which is finite in each component of~$\Vbarbb$ because~$T$ has degree~1 and~$x \in \Vbarbb$ is zero in sufficiently low degree.
\begin{proposition}\label{proposition:MuAndTranslations}
  Let~$b_1,\ldots,b_n \in \V$ and~$w \in \C^n \setminus\Delta$. For all~$z\in\C$,
  \begin{align}
    e^{zT}\mu_Y(b_1,w_1,\ldots,b_n,w_n) = \mu_Y(b_1,w_1+z,\ldots,b_n, w_n +z) \displayperiod\label{display:MuAndTranslations}
  \end{align}
\end{proposition} 
In the case~$n=1$ and~$w_1=0$, the proposition together with insertion at zero implies that~$e^{zT}b = Y(b,z)\vac$ for~$b\in \V$ and~$z\in \C$. This means that~$a_{(k)}\vac = \frac{1}{(-k-1)!}T^{-k-1}a$ for~$k\leq -1$, which one could have also deduced directly from the translation axiom of the vertex algebra~$\V$.
In particular~$Ta = a_{(-2)}\vac$ and~$\mu_Y(a,w) = a + Ta w + \ldots$.
Also, since~$e^{zT}a = \mu_Y(a,z)$ for all~$a\in \V$ and~$z\in \C$, the proposition is the special case~$m=0$ of the associativity axiom for~$\mu_Y$.
\begin{proof} 
  Both sides of the equation are holomorphic functions of~$(z,w) \in \C \times (\C^n \setminus\Delta)$.
  By the identity theorem we may assume 
  \begin{align*}
    |w_1| > \ldots > |w_n|,
  \end{align*}
  and having fixed a choice of~$w$, we may assume~$z$ to be in the non-empty, open, and convex region~$B$ of those~$z$ such that 
  \begin{align*}
    |w_1 + z| > \ldots > |w_n +z|.
  \end{align*}
  For~$z\in B$, the equation is equivalent to
  \[
    e^{zT}Y(b_1,w_1)\ldots Y(b_n,w_n)\vac = Y(b_1,w_1+z)\ldots Y(b_n, w_n +z)\vac\displayperiod
  \]
  Let~$f(z)$ equal the left-hand side and~$g(z)$ equal the right-hand side for~$z\in B$.
  The function~$f$ is the unique solution of the (holomorphic) initial value problem
  \begin{align}
    \partial_z \varphi (z) =T\varphi(z)\quad\varphi(0) = Y(b_1,w_1)\ldots Y(b_n, w_n)\vac \quad \varphi \in \mathcal{O}(B;\Vbarbb) \label{display:TranslationInvarianceIVP}
  \end{align}
  as can be deduced from the fact that~$\varphi(z) = e^{zA}x$ is the unique solution of 
  \begin{align}
    \partial_z \varphi (z) =A\varphi(z)\quad\varphi(0) = x \quad \varphi \in \mathcal{O}(B;X)\displayperiod \label{display:TranslationInvarianceIVPFiniteDimensionalVersion}  
  \end{align}
  (Here,~$X$ is not a finite-dimensional subspace of~$\V_k$ for~$k\in\Z$, but instead ranges over a sequence of appropriately chosen finite-dimensional subquotients of~$\Vbarbb$.)

  We prove by induction on~$n$ that~$g$ is a solution of~(\ref{display:TranslationInvarianceIVP}), too.
  The base case~$n=0$ holds because then~$g(z) = \vac$.
The induction hypothesis implies
  \[
    \partial_zY(b_2,w_2+z)\ldots Y(b_n, w_n +z)\vac = TY(b_2,w_2+z_2)\ldots Y(b_n,w_n+z)\vac\displayperiod
  \]
  Proposition~\ref{proposition:ProductOfYs} implies
  \begin{align*}
    g(z) &= \sum_{k\in \Z} {b_1}_{(k)} (w_1+z)^{-k-1} Y(b_2,w_2+z)\ldots Y(b_m, w_m +z)\vac\\
  \intertext{and thus}
    \partial_z g(z) &= \sum_{k \in \Z} {b_1}_{(k)} (-k-1) (w_1+z)^{-k-2} Y(b_2,w_2+z)\ldots Y(b_m, w_m +z)\vac\\
    &\phantom{=} + \sum_{k\in \Z} {b_1}_{(k)} (w_1+z)^{-k-1} \partial_z  Y(b_2,w_2+z)\ldots Y(b_m, w_m +z)\vac\\
    &= \sum_{k \in \Z} {b_1}_{(k-1)} (-k) (w_1+z)^{-k-1} Y(b_2,w_2+z)\ldots Y(b_m, w_m +z)\vac\\
    &\phantom{=}+ \sum_{k\in \Z} {b_1}_{(k)} (w_1+z)^{-k-1} T Y(b_2,w_2+z)\ldots Y(b_m, w_m +z)\vac
  \end{align*}
  after reindexing.
  The translation axiom implies that~$[T,{b_1}_{(k)}] = -k{b_1}_{(k-1)}$.
  Therefore
  \begin{align*}
    \partial_z g(z) &= \sum_{k \in \Z} T{b_1}_{(k)} (w_1+z)^{-k-1} Y(b_2,w_2+z)\ldots Y(b_m, w_m +z)\vac \\
                    &= T Y(b_1,w_1+z)\ldots Y(b_m, w_m +z)\vac = Tg(z),
  \end{align*}
  and this concludes the induction step. 
\end{proof}
\begin{proposition}\label{proposition:MuIsAssociative}
$\mu_Y$ satisfies associativity.
\end{proposition}
\begin{proof}
  Let~$a_1,\ldots,a_{m+1},b_1,\ldots,b_n \in \V$.
  Associativity is about a normally convergent sum of holomorphic functions of~$
  (z,w)\in (\C^{m+1}\setminus\Delta)\times(\C^n\setminus\Delta)$ with $\max_j |w_j| < \min_{1\leq i\leq m} |z_{m+1} - z_i|$.
  First, we prove a version of associativity for~$z_{m+1}$ fixed and equal to zero.
  The set of~$(z,w) \in (\C^{m}\setminus\Delta)\times (\C^n\setminus\Delta)$ such that~$\max_j |w_j| < \min_i |z_i|$ is covered by the set of open subsets of the form
  \[
    U_{r,R} = \{ (z,w) \in \C^{m}\setminus\Delta\times\C^n\setminus\Delta \mid \max_j |w_j| < r < R <\min_i |z_i|\}
  \]
  where~$0<r<R$, so it suffices to prove normal convergence and the equation in the associativity axiom on~$U_{r,R}$. 
  Let
  \[
    W_{r,R} = U_{r,R} \times A_{0,R/r} \quad \text{where } A_{0,R/r} = \{ \lambda \in \C^\times \mid 0 < |\lambda| < R/r \}\displaycomma
  \]
  and let~$f \in \mathcal{O}(W_{r,R};\Vbarbb)$ be the holomorphic function defined by
  \begin{align}
    f(z_1,\ldots,z_m, w_1,\ldots, w_n, \lambda) = \mu_Y(a_1,z_1,\ldots, a_m, z_m, \lambda.b_1, \lambda w_1, \ldots, \lambda.b_n, \lambda w_n)\displayperiod\notag
  \end{align}
  We consider the Laurent expansion
  \begin{align}
    f(z_1,\ldots,z_m, w_1,\ldots, w_n, \lambda) = \sum_{k\in \Z} c_k(z_1,\ldots,z_m,w_1,\ldots,w_n) \lambda^k\label{equation:SumOverLambda}
  \end{align}
  on~$A_{0,R/r}$ where~$c_k \in \mathcal{O}(U_{r,R};\Vbarbb)$ such that
  \begin{align*}
  c_k(z_1,\ldots,z_m,w_1,\ldots,w_n) &= \ootpi\int_{S^1} f(z_1,\ldots,z_m, w_1,\ldots, w_n,\lambda)\lambda^{-k-1} d\lambda
  \end{align*}
  for all~$k\in\Z$.
  It suffices to consider the case that~$U_{r,R}$ is non-empty, and, in this case, the subset of~$(z,w)\in U_{r,R}$ such that $|z_1|>\ldots>|z_m|>|w_1|>\ldots>|w_n|$ is open and non-empty. For such~$(z,w)$,
\begin{align*}
&c_k(z_1,\ldots,z_m,w_1,\ldots,w_n)\\
&=\ootpi\int_{S^1} Y(a_1,z_1)\ldots Y(a_m,z_m)Y(\lambda.b_1,\lambda w_1)\ldots Y(\lambda.b_n,\lambda w_n)\vac \lambda^{-k-1} d\lambda \\
&=Y(a_1,z_1)\ldots Y(a_m,z_m) \ootpi\int_{S^1} Y(\lambda.b_1,\lambda w_1)\ldots Y(\lambda.b_n,\lambda w_n)\vac \lambda^{-k-1} d\lambda \\
&=Y(a_1,z_1)\ldots Y(a_m,z_m)\ootpi\int_{S^1} \mu_Y(\lambda.b_1,\lambda w_1,\ldots,\lambda.b_n,\lambda w_n) \lambda^{-k-1} d\lambda \\
&=Y(a_1,z_1)\ldots Y(a_m,z_m)\ootpi \int_{S^1} \lambda.\mu_Y(b_1,w_1,\ldots,b_n,w_n) \lambda^{-k-1} d\lambda \\
&=Y(a_1,z_1)\ldots Y(a_m,z_m) p_k\mu_Y(b_1,w_1,\ldots,b_n,w_n)\\
&=\mu_Y(a_1,z_1,\ldots,a_m,z_m,p_k\mu_Y(b_1,w_1,\ldots,b_n,w_n),0)
\end{align*}
so, by the identity theorem on the domain~$U_{r,R}$, the first and last expression agree for all~$(z,w) \in U_{r,R}$.
Evaluating the normally convergent Laurent expansion in Equation~(\ref{equation:SumOverLambda}) at~$\lambda=1$ yields the normally convergent sum
\begin{align*}
&\mu_Y(a_1,z_1,\ldots,a_m,z_m, b_1,w_1,\ldots,b_n,w_n)=f(z_1,\ldots,z_m,w_1,\ldots,w_n,1)\\
&= \sum_{k\in \Z} \mu_Y(a_1,z_1,\ldots,a_m,z_m,p_k\mu_Y(b_1,w_1,\ldots,b_n,w_n),0)\displayperiod
\end{align*}
This finishes the proof of the version of associativity for~$z_{m+1} = 0$.

Before we treat~$z_{m+1}$ as a variable, we prove that there is an integer~$H$ depending on~$a_1,\ldots,a_m,b_1,\ldots,b_n$ but not on~$z,w,k$ such that
\begin{align}
p_h\mu_Y(a_1,z_1,\ldots,a_m,z_m,p_k\mu_Y(b_1,w_1,\ldots,b_n,w_n),0) = 0\label{equation:phZero}
\end{align}
for all integers~$h<H$.
Without loss of generality, we may assume that~$b_1,\ldots,b_n$ are homogeneous.
There is an integer~$H$ indepedent of~$z,w,\lambda$ such that
\begin{align*}
&p_h(\mu_Y(a_1,z_1,\ldots,a_m,z_m, \lambda.b_1, \lambda w_1,\ldots,\lambda .b_n,\lambda w_n))\\
&=\lambda^{|b_1|+\ldots+|b_n|} p_h(\mu_Y(a_1,z_1,\ldots,a_m,z_m,b_1, \lambda w_1,\ldots,b_n,\lambda w_n))=0
\end{align*}
for all integers~$h<H$ by Proposition~\ref{proposition:ProductOfYs}.
This implies that the integrand in
\begin{align*}
&p_h\mu_Y(a_1,z_1,\ldots,a_m,z_m,p_k\mu_Y(b_1,w_1,\ldots,b_n,w_n),0)\\
&=\ootpi \int_{S^1} p_h(\mu_Y(a_1,z_1,\ldots,a_m,z_m, \lambda.b_1, \lambda w_1,\ldots,\lambda .b_n,\lambda w_n))\lambda^{-k-1}d \lambda
\end{align*}
is zero for all integers~$h<H$, so Equation~(\ref{equation:phZero}) follows.

We now treat~$z_{m+1}$ as a variable.
The case of~$z_{m+1} = 0$ implies that
\begin{align}
  &\sum_{k\in\Z} \mu_Y(a_1,z_1-z_{m+1},\ldots,a_m,z_m-z_{m+1},p_k\mu_Y(b_1,w_1,\ldots,b_n,w_n),0)\notag\\
  &=\mu_Y(a_1,z_1-z_{m+1},\ldots,a_m,z_m-z_{m+1},b_1,w_1,\ldots,b_n,w_n)\label{equation:AssociativityForZ1MinusZMPlus1}
\end{align}
with normal convergence as a sum of functions of~$z_1,\ldots,z_m,z_{m+1},w_1,\ldots,w_n$ because the subtraction map~$\C\times\C \rightarrow \C$ is continuous.
We use the abbreviations
\begin{align*}
&b=b_1\otimes\ldots\otimes b_n,\quad  w=(w_1,\ldots,w_n),\quad \mu_Y(b,w) = \mu_Y(b_1,w_1,\ldots,b_n,w_n)
\end{align*}
in the following.
Let~$l\in\Z$.
Proposition~\ref{proposition:MuAndTranslations} and Equation~(\ref{equation:phZero}) imply
\begin{align*}
&p_l\mu_Y(a_1,z_1,\ldots,a_m,z_m,p_k\mu_Y(b_1,w_1,\ldots,b_n,w_n),z_{m+1})\\
&= p_l e^{z_{m+1}T}\mu_Y(a_1,z_1-z_{m+1},\ldots,a_m,z_m-z_{m+1},p_k\mu_Y(b,w),0)\\
&= p_l \sum^\infty_{i=0} \frac{z^i_{m+1}}{i!}T^i\mu_Y(a_1,z_1-z_{m+1},\ldots,a_m,z_m-z_{m+1},p_k\mu_Y(b,w),0)\\
&= \sum^{l-H}_{i=0} \frac{z^i_{m+1}}{i!}T^i p_{l-i}\mu_Y(a_1,z_1-z_{m+1},\ldots,a_m,z_m-z_{m+1},p_k\mu_Y(b,w),0).
\end{align*}
Applying the homogeneous operator~$T$ of degree~1, multiplying with the locally bounded function~$z^i_{m+1}$ of~$z_{m+1}$, and taking the finite sum over~$i$ preserves normal convergence, so Equation~(\ref{equation:AssociativityForZ1MinusZMPlus1}) with~$l-i$ substituted for~$l$ implies that
\begin{align*}
  \sum_{k\in\Z} p_l\mu_Y(a_1,z_1,\ldots,a_m,z_m,p_k\mu_Y(b_1,w_1,\ldots,b_n,w_n),z_{m+1})
\end{align*}
is normally convergent and equal to 
\begin{align*}
  &\sum^{l-H}_{i=0} \frac{z^i_{m+1}}{i!}T^i \sum_{k\in\Z} p_{l-i}\mu_Y(a_1,z_1-z_{m+1},\ldots,a_m,z_m-z_{m+1},p_k\mu_Y(b,w),0)\\
  &=\sum^{l-H}_{i=0} \frac{z^i_{m+1}}{i!}T^i p_{l-i}\mu_Y(a_1,z_1-z_{m+1},\ldots,a_m,z_m-z_{m+1},b_1,w_1,\ldots,b_n,w_n)\\
  &= p_l e^{z_{m+1}T} \mu_Y(a_1,z_1-z_{m+1},\ldots,a_m,z_m-z_{m+1},b_1,w_1,\ldots,b_n,w_n)\\
  &= p_l\mu_Y(a_1,z_1,\ldots,a_m,z_m,b_1,w_1+z_{m+1},\ldots,b_n,w_n+z_{m+1}).
\end{align*}
\end{proof} 
\bibliographystyle{plain}
\bibliography{geom_va}  

\begin{thebibliography}{1}

\bibitem{CostelloGwilliam}
Kevin Costello and Owen Gwilliam.
\newblock {\em Factorization algebras in quantum field theory. {V}ol. 1},
  volume~31 of {\em New Mathematical Monographs}.
\newblock Cambridge University Press, Cambridge, 2017.

\bibitem{FrenkelBenZvi}
Edward Frenkel and David Ben-Zvi.
\newblock {\em Vertex algebras and algebraic curves}, volume~88 of {\em
  Mathematical Surveys and Monographs}.
\newblock American Mathematical Society, Providence, RI, second edition, 2004.

\bibitem{FrenkelLepowskyMeurman}
Igor Frenkel, James Lepowsky, and Arne Meurman.
\newblock {\em Vertex operator algebras and the {M}onster}, volume 134 of {\em
  Pure and Applied Mathematics}.
\newblock Academic Press, Inc., Boston, MA, 1988.

\bibitem{FrenkelHuangLepowsky}
Igor~B. Frenkel, Yi-Zhi Huang, and James Lepowsky.
\newblock On axiomatic approaches to vertex operator algebras and modules.
\newblock {\em Mem. Amer. Math. Soc.}, 104(494):viii+64, 1993.

\bibitem{GunningRossi}
Robert~C. Gunning and Hugo Rossi.
\newblock {\em Analytic functions of several complex variables}.
\newblock AMS Chelsea Publishing, Providence, RI, 2009.
\newblock Reprint of the 1965 original.

\bibitem{HuangTwoDimConfGeomAndVOAs}
Yi-Zhi Huang.
\newblock {\em Two-dimensional conformal geometry and vertex operator
  algebras}, volume 148 of {\em Progress in Mathematics}.
\newblock Birkh\"{a}user Boston, Inc., Boston, MA, 1997.

\bibitem{Remmert}
R.~Remmert.
\newblock {\em Theory of Complex Functions}.
\newblock Graduate Texts in Mathematics. Springer New York, 1991.

\bibitem{RunkelCFTandVOAs}
Ingo Runkel.
\newblock Conformal field theory and vertex operator algebras.
\newblock In {\em Operator algebras and mathematical physics}, volume~80 of
  {\em Adv. Stud. Pure Math.}, pages 1--22. Math. Soc. Japan, Tokyo, 2019.

\end{thebibliography}
\end{document}